\documentclass[11pt]{article}

\usepackage{amsthm,amsmath,amssymb}

\input xypic
\xyoption{graph}
\xyoption{curve}
\newtheorem{theo}[subsubsection]{Theorem}
\newtheorem{lem}[subsubsection]{Lemma}
\newtheorem{prop}[subsubsection]{Proposition}
\newtheorem{cor}[subsubsection]{Corollary}
\newtheorem{exam}[subsubsection]{Example}
\newtheorem{rem}[subsubsection]{Remark}
\newtheorem{defi}[subsubsection]{Definition}

\newcommand{\bS}{{\bf S}}
\newcommand{\bP}{{\bf P}}

\newcommand{\B}{{\bf B}}

\newcommand{\cC}{{\cal C}}
\newcommand{\D}{{\cal D}}

\newcommand{\cT}{{\cal Z}}

\newcommand{\id}{{\it id}}
\newcommand{\Gr}{{\bf Gr}}
\newcommand{\Sets}{{{\bf Sets}}}
\newcommand{\Vect}{{\bf Vect}}
\newcommand{\Vines}{{\bf Vines}}
\newcommand{\Mon}{{\bf Mon}}
\newcommand{\Mod}{{\bf Mod}}
\newcommand{\Moncat}{{\bf Moncat}}
\newcommand{\Cat}{{\bf Cat}}
\newcommand{\Common}{{\bf Common}}

\newbox\Treeone
\setbox \Treeone=\hbox{\xygraph{ !{0;/r1.0pc/:;/u1.0pc/::}[]
   -[u]
   -[u]
   }}
\newcommand*{\TreeOne}{\copy\Treeone}

\newbox\Treeoneone
\setbox\Treeoneone =\hbox{\xygraph{ !{0;/r1.0pc/:;/u1.0pc/::}[]
  ( -[ul]
    -[u]
  , -[ur]
    -[u]
  )}}
\newcommand*{\TreeOneOne}{\copy\Treeoneone}

\newbox\Treetwo
\setbox\Treetwo =\hbox{\xygraph{ !{0;/r1.0pc/:;/u1.0pc/::}[]
   -[u]
  ( -[ul]
  , -[ur]
  )}}
\newcommand*{\TreeTwo}{\copy\Treetwo}

\newbox\Treethree
\setbox\Treethree =\hbox{\xygraph{ !{0;/r1.0pc/:;/u1.0pc/::}[]
   -[u]
  ( -[ul]
  , -[u]
  , -[ur]
  )}}
\newcommand*{\TreeThree}{\copy\Treethree}

\newbox\Treeoneoneone
\setbox\Treeoneoneone =\hbox{\xygraph{ !{0;/r1.0pc/:;/u1.0pc/::}[]
  ( -[ul]
    -[u]
  , -[u]
    -[u]
  , -[ur]
    -[u]
  )}}
\newcommand*{\TreeOneOneOne}{\copy\Treeoneoneone}

\newbox\Treetwoone
\setbox\Treetwoone =\hbox{\xygraph{ !{0;/r1.0pc/:;/u1.0pc/::}[]
  ( -[ul]
    ( -[ul]
    , -[ur]
    )
  , -[ur]
    -[u]
  )}}
\newcommand*{\TreeTwoOne}{\copy\Treetwoone}

\newbox\Treeonetwo
\setbox\Treeonetwo =\hbox{\xygraph{ !{0;/r1.0pc/:;/u1.0pc/::}[]
  ( -[ul]
    -[u]
  , -[ur]
    ( -[ul]
    , -[ur]
    )
  )}}
\newcommand*{\TreeOneTwo}{\copy\Treeonetwo}

\begin{document}
\author{A. Davydov}
\title{Quasi-commutative algebras}
\maketitle
\begin{abstract}
We characterise algebras commutative with respect to a Yang-Baxter operator (quasi-commutative algebras) in terms of
certain cosimplicial complexes. In some cases this characterisation allows the classification of all possible
quasi-commutative structures.
\end{abstract}
\tableofcontents
\section{Introduction}
From a category theory point of view the main result of the paper is a characterisation of the free braided monoidal
category, containing a commutative monoid, as a localisation of some combinatorially given category. The situation in
the (symmetric) monoidal case is well understood (as basic examples of algebraic theories \cite{lw} or PROPs
\cite{mc0}). For example, the free monoidal category, containing a monoid, is the category of finite ordered sets and
order preserving maps; the free symmetric monoidal category, containing a commutative monoid, is the category of finite
sets; the free symmetric monoidal category, containing a monoid, is the category of finite sets and maps with linear
orders on fibres. In all these examples the tensor product is given by the disjoint union of sets.
\newline
The braided case seems to be much less combinatorial. Usually, the free braided monoidal categories, containing some
sort of monoid, are given by (monoidal) generators and relations (as free monoidal categories, containing some more
complicated algebraic structure) or as categories of geometric objects (geometric braids, vines etc.).
\newline
In this paper we identify the free braided monoidal category, containing a commutative monoid, with the free monoidal
category containing a cosimplicial monoid, satisfying some additional condition (the one we call the {\em covering
condition}). This condition requires the invertibility of certain ({\em covering}) maps, composed of tensor products of
cosimplicial maps and monoid multiplications. We also present a combinatorial model for the free monoidal category
containing a cosimplicial monoid. It is the category of pruned trees of height 2 in the sense of M. Batanin
\cite{ba,bs}, which can be seen as ordered sets of ordered sets. Under this identification covering maps correspond to
maps of trees, which are bijective on leaves. Thus our main result says that the free braided monoidal category,
containing a commutative monoid, is the category of pruned trees of height 2, localised with respect to maps, bijective
on leaves. We also note that the collection of maps bijective on leaves is monoidally generated by just two maps
(between trees with two leaves).

From an algebraic point of view the main result provides a Yang-Baxter operator on a monoid (algebra) if it is the
first component of a cosimplicial monoid (algebra) satisfying the covering condition. Actually, the whole cosimplicial
monoid is not necessary, only its first three components. Known Yang-Baxter operators on monoid-like objects fit into
this scheme (groups, Hopf and Hopf-Galois algebras). In the case of groups  we recover the characterisation of groups,
commutative with respect to a Yang-Baxter operator, in terms of bijective 1-cocycles (see \cite{ess}).

\section*{Acknowledgement}
The author would like to thank members of the Australian Category Seminar for fruitful discussions. Special thanks
are to M. Batanin for introducing me into the language of trees,  S. Lack for
referring me to \cite{la}, R. Moore and D. Steffen for numerous consultations in xy-pic and to R. Street for encouraging
the author to look beyond, when he was struggling to prove Yang-Baxter axiom using just the first two components of
the cosimplicial object (as was said by Kung Fu Tzu ``it is very hard to find a black cat in a dark room,
especially if it is not there").

\section{Commutative monoids in braided categories}
Here we give list some properties of and constructions for (commutative) monoids in (braided, symmetric) monoidal
categories. The results of this section are mostly well-known. Sometimes we add (sketches of) proofs.

By a monoidal functor we will mean a strong monoidal functor (monoidal structure constraints $F(X)\otimes F(Y)\to
F(X\otimes Y), 1\to F(1)$ are assumed to be isomorphisms). If not stated otherwise we assume monoidal categories to be
strict (associativity and unit constraints are given by identity morphisms).

\subsection{Monoids in monoidal categories}\label{mmc}
Here we recall (see \cite{mc}) the description of a free monoidal category generated by a monoid and its relation to
(co)simplicial complexes.

Let $\cC$ be a strict monoidal category. A {\em monoid} in $\cC$ is an object $A$ with morphisms $$\mu:A\otimes A\to A\
\mbox{(multiplication map)},$$ $$\iota:1\to A\ \mbox{(unit inclusion)}$$ satisfying the conditions $$\mu(\mu\otimes I)
= \mu(I\otimes\mu),$$ $$\mu(\iota\otimes I) = \mu(I\otimes\iota) = I.$$ {\em A homomorphism of monoids} is a morphism
$f:A\to B$ such that $$\mu(f\otimes f) = f\mu,\quad f\iota = \iota.$$

Monoids in a monoidal category $\cC$ form a category which we will denote $\Mon(\cC)$. A monoidal functor $\cC\to\D$
induces a functor $\Mon(\cC)\to\Mon(\D)$. This amounts to a 2-functor $\Mon:\Moncat\to\Cat$ from the 2-category
$\Moncat$ of monoidal categories with monoidal functors and monoidal natural transformations to the 2-category $\Cat$
of categories.

The 2-functor $\Mon$ is representable, i.e. there
is a monoidal category $\Delta$ with a monoid $A\in\Delta$ ({\em the free monoidal category containing a monoid})
such that the evaluation functor at $A$ from the category of monoidal functors $\Moncat(\Delta,\cC)$ into the category
of monoids $\Mon(\cC)$ is an equivalence. The category $\Delta$ admits the following explicit description: it is the
(skeletal) category of well-ordered finite sets with order preserving maps. Objects of $\Delta$ are parameterised by
natural numbers, we will use the following notation $[n] = \{1,...,n\}$ for an $n$-element set (different to MacLane's). Tensor product on $\Delta$
is given by the sum (the ordered union) of sets. Obviously the sum of order-preserving maps is order-preserving. The unit
object for this tensor product is the empty set $[0]$. The monoid $A\in\Delta$ is the one-element set $[1]$ with monoid
structure given by the unique maps $\mu:[1]\otimes[1] = [2]\to[1]$ and $\iota:[0]\to[1]$.

Below we establish the freeness of $\Delta$. The functor $$\Mon(\cC)\to\Moncat(\Delta,\cC),\ A\mapsto F_A$$ adjoint
(quasi-inverse) to the evaluation functor $$\Moncat(\Delta,\cC)\to\Mon(\cC),\ F\mapsto F[1]$$ can be constructed as
follows. The value $F_A[n]$ is the $n$-th power $A^{\otimes n}$. For an order-preserving map $f:[n]\to [m]$ the
morphism $F_A(f):A^{\otimes n}\to A^{\otimes m}$ is the tensor product
$\mu(|f^{-1}(1)|)\otimes...\otimes\mu(|f^{-1}(m)|)$ where $|f^{-1}(i)|$ is the cardinality of the fibre
$f^{-1}(i)=\{j\in[n]: f(j)=i\},$ and where $\mu(0) = \iota:1\to A$, $\mu(1) = I:A\to A$, and $\mu(k):A^{\otimes k}\to
A$ is the iterated multiplication $\mu(\mu\otimes I)...(\mu\otimes I\otimes...\otimes I)$ for $k>1$.

The category $\Delta$ has the following presentation. Denote by $\sigma^n_i:[n]\to [n-1], i=0,...,n-2$, surjective
monotone maps, defined by the condition $\sigma^n_i(i) = \sigma^n_i(i+1)$, and by $\partial^n_j:[n]\to[n+1],
j=0,...,n,$ injective monotone maps, not taking the value $i$. It is not hard to see that these maps satisfy the
identities (here and later on we omit the superscripts): $$\begin{array}{llll} \sigma_j\sigma_i & = &
\sigma_i\sigma_{j+1}, & i\leq j \\
\partial_i\partial_j & = & \partial_{j+1}\partial_i, & i\leq j \\ \sigma_j\partial_i & = & \partial_i\sigma_{j-1}, &
i<j \\ & = & 1, & i=j,j+1\\ & = & \partial_{i-1}\sigma_j, & i>j+1
\end{array}$$
The morphisms of the category $\Delta$ are generated by $\sigma^n_i$ and $\partial^n_j$ subject to the identities above
\cite{mc}.

The above presentation allows the following construction. With a monoid we can associate a cosimplicial object $A^*$ in
$\cC$ ({\em cobar construction}): where $A^n = A^{\otimes n}$, with coface maps $\sigma_i:A^n\to A^{n-1}$ defined by
$\sigma_i = I^{\otimes i}\otimes\mu\otimes I^{\otimes n-i-2}$ and codegeneration maps $\partial_i:A^n\to A^{n+1}$
defined by $\partial_i = I^{\otimes i}\otimes\iota\otimes I^{\otimes n-i}$.

\subsection{Monoids in braided monoidal categories}\label{mbc}

Here we recall the well-know properties of monoids in a braided strict monoidal category $\cC$ (see \cite{js}).

\begin{lem}\label{monm}
The product $A\otimes B$ of two monoids in a braided category can be equipped with the structure of a
monoid:
\begin{equation}\label{prpr}
\mu_{A\otimes B} = (\mu_A\otimes\mu_B)(A\otimes c_{B,A}\otimes B),\quad \iota_{A\otimes B} = \iota_A\otimes\iota_B
\end{equation}
making the category of monoids $\Mon(\cC)$ a monoidal category.
\end{lem}
\begin{proof}
Associativity of the multiplication (\ref{prpr}) follows from the commutative diagram:

$$\xygraph{ !{0;/r6.0pc/:;/u4.0pc/::}[]*+{(A\otimes B)^{\otimes 3}} (
 ( :[r(6)]*+{(A\otimes B)^{\otimes 2}}="(3,3)" ^{A\otimes B\otimes\mu_{A\otimes B}}
   ( :[d(6)]*+{A\otimes B}="(3,-3)" ^{\mu_{A\otimes B}}
   , :[d(3)l]*+{A^{\otimes 2}\otimes B^{\otimes 2}}="(2,0)" _{A\otimes c_{B,A}\otimes B}
     :"(3,-3)" _{\mu_A\otimes\mu_B}
   )
 , :[r(3)d]*+{A\otimes B\otimes A^{\otimes 2}\otimes B^{\otimes 2}} ^{A\otimes B\otimes A\otimes c_{B,A}\otimes B}
   ( :"(3,3)" ^{A\otimes B\otimes\mu_A\otimes\mu_B}
   , :[dd]*+{A^{\otimes 3}\otimes B^{\otimes 3}}="(0,0)" ^{A\otimes c_{B,A^{\otimes 2}}\otimes B^{\otimes 2}}
     ( :"(2,0)" _{A\otimes\mu_A\otimes B\otimes\mu_B}
     , :[dd]*+{A^{\otimes 2}\otimes B^{\otimes 2}}="(0,-2)" ^{\mu_A\otimes A\otimes\mu_B\otimes B}
       :"(3,-3)" ^{\mu_A\otimes\mu_B}
     )
   , :[dl]*+{A^{\otimes 2}\otimes B\otimes A\otimes B^{\otimes 2}}="(-1,1)"
      _{A\otimes c_{B,A}\otimes A\otimes B^{\otimes 2}}
     :"(0,0)" ^{A^{\otimes 2}\otimes c_{B,A}\otimes B^{\otimes 2}}
   )
 , :[d(3)r]*+{A^{\otimes 2}\otimes B^{\otimes 2}\otimes A\otimes B} ^{A\otimes c_{B,A}\otimes A\otimes A\otimes B}
   ( :"(-1,1)" ^{A^{\otimes 2}\otimes B\otimes c_{B,A}\otimes B}
   , :"(0,0)" _>>>>>>>>>>>>{A^{\otimes 2}\otimes c_{B^{\otimes 2},A}\otimes B}
   , :[d(3)l]*+{(A\otimes B)^{\otimes 2}}="(-3,-3)" ^{\mu_A\otimes\mu_B\otimes A\otimes B}
     ( :"(0,-2)" ^{A\otimes c_{B,A}\otimes B}
     , :"(3,-3)" _{\mu_{A\otimes B}}
     )
   )
 , :"(-3,-3)" _{\mu_{A\otimes B}\otimes A\otimes B}
 )
) }$$

Analogously tensor product of homomorphisms of monoids is a homomorphism of tensor products: 
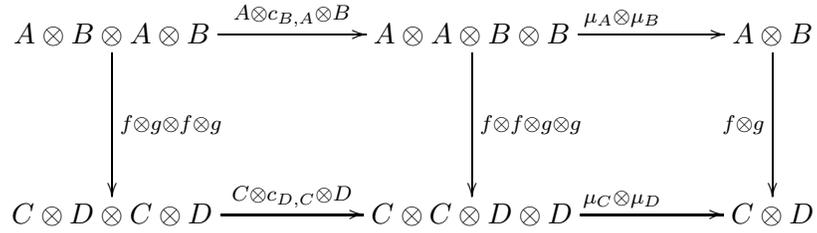
\xymatrix{A\otimes
B\otimes A\otimes B \ar[rr]^{A\otimes c_{B,A}\otimes B} \ar[dd]^{f\otimes g\otimes f\otimes g} & & A\otimes
A\otimes B\otimes B \ar[rr]^{\mu_A\otimes\mu_B} \ar[dd]^{f\otimes f\otimes g\otimes g} & & A\otimes B
\ar[dd]_{f\otimes g} \\ \\ C\otimes D\otimes C\otimes D \ar[rr]^{C\otimes c_{D,C}\otimes D} & & C\otimes C\otimes
D\otimes D \ar[rr]^{\mu_C\otimes\mu_D} & & C\otimes D }

Finally, associativity of the tensor product $(A,\mu_A)\otimes(B,\mu_B) = (A\otimes B,\mu_{A\otimes B})$ follows
from the axioms of braiding. Indeed, the equality $$(A\otimes B\otimes c_{C,B})(c_{B\otimes C,A}\otimes B) =
(c_{B,A}\otimes B\otimes C)(B\otimes c_{C,A\otimes B})$$ implies that $$\mu_{A\otimes(B\otimes C)} =
(\mu_A\otimes\mu_B\otimes\mu_C)(A\otimes A\otimes B\otimes c_{C,B}\otimes C)(A\otimes c_{B\otimes C,A}\otimes
B\otimes C)$$ coincides with $$\mu_{(A\otimes B)\otimes C} = (\mu_A\otimes\mu_B\otimes\mu_C)(A\otimes
c_{B,A}\otimes B\otimes C\otimes C)(A\otimes B\otimes c_{C,A\otimes B}\otimes C).$$
\end{proof}

The cobar complex $A^*$ of a monoid $A$ in a braided monoidal category has the following extra bit of structure coming
from braiding. Each $A^n$ is equipped with the action of the braid group $B_n$, intertwined with the cosimplicial maps
in the following way:
\begin{lem}\label{bractcc}
Let $x_i$ denote the generator of $B_n$, acting on $A^n$ by $I_{A^{i-1}}\otimes c_{A,A}\otimes I_{A^{n-i-1}}$. Then
$$x_i\sigma_j = \left\{ \begin{array}{ll} \sigma_j x_i, & i<j \\ \sigma_{i+1}x_ix_{i+1}, & i=j\\ \sigma_ix_{i+1}x_i, &
i=j-1\\ \sigma_jx_{i+1}, & i>j-1 \end{array}\right. \quad\quad x_i\partial_j = \left\{ \begin{array}{ll} \partial_j x_i, &
i<j-1
\\ \partial_{i-1}, & i=j\\ \partial_{i+1}, & i=j-1\\ \partial_jx_{i-1}, & i>j-1 \end{array}\right.$$
\end{lem}
\begin{proof}
The first and the last equations in both cases follow from functoriality of the tensor product (``sliding property").
The two middle equations in the first case are consequences of the commutativity of the diagrams: $$\xymatrix{ &
A^{\otimes 3} \ar[rr]^{\mu A} \ar[dd]^{c_{A^{\otimes 2},A}} \ar[ld]_{Ac_{A,A}} && A^{\otimes 2} \ar[dd]_{c_{A,A}} & &
A^{\otimes 3} \ar[rr]^{A\mu} \ar[dd]^{c_{A,A^{\otimes 2}}} \ar[ld]_{c_{A,A}A} && A^{\otimes 2} \ar[dd]_{c_{A,A}}\\
A^{\otimes 3} \ar[dr]_{c_{A,A}A} &&& &A^{\otimes 3} \ar[dr]_{Ac_{A,A}} \\ & A^{\otimes 3} \ar[rr]^{A\mu}
 && A^{\otimes 2} && A^{\otimes 3} \ar[rr]^{\mu A} && A^{\otimes 2} }$$
Similarly the two middle equations in the first case follow from the commutative diagrams: $$\xymatrix{ A
\ar[rr]^{\iota A} \ar[drr]_{A\iota} && A^{\otimes 2} \ar[d]^{c_{A,A}} && A \ar[rr]^{A\iota} \ar[drr]_{\iota A} &&
A^{\otimes 2} \ar[d]^{c_{A,A}} \\ && A^{\otimes 2} &&& & A^{\otimes 2} }$$
\end{proof}
The above lemma defines a {\em distributive law} $$\Delta(m,n)\times\B(n,n)\to \B(m,m)\times\Delta(m,n)$$ for order
preserving maps over braids. Here $\B(n,n)$ is the braid group on $n$ strings. This distributive law was used in
\cite{ds} to characterise the free braided monoidal category containing a monoid as a mixture $\B\Delta$ of the free
braided category on one object $\B$ and the free monoidal category containing a monoid $\Delta$. Objects of $\B\Delta$
are natural numbers. Morphisms of $\B\Delta$ are pairs $\B\Delta(m,n) = \B(m,m)\times\Delta(m,n)$, with composition
defined by means of the distributive law and compositions in $\B$ and $\Delta$. In other words we have the following.
\begin{cor}\label{dist}
The natural monoidal functors from $\B$ and $\Delta$ to the free braided monoidal category, containing a monoid, are
isomorphic on objects.  Any morphism of the free braided monoidal category, containing a monoid, can be uniquely
decomposed into a morphism in $\B$ followed by a morphism in $\Delta$.
\end{cor}

A monoid $(A,\mu,\iota)$ in $\cC$ is {\em commutative} if $$\mu c_{A,A} = \mu.$$
\begin{lem}\label{comm}
For a commutative monoid $A$, the multiplication map $\mu:A\otimes A\to A$ is a homomorphism of monoids:
\end{lem}
\begin{proof}
This follows from the commutative diagram: \xymatrix{ A^{\otimes 4} \ar[dd]_{\mu_{A^{\otimes 2}}}
\ar[rrr]^{\mu\otimes\mu} \ar[dr]_{1\otimes c\otimes 1} \ar[drr]^{1\otimes\mu\otimes 1} & & & A^{\otimes 2}
\ar[dd]^{\mu} \\ & A^{\otimes 4} \ar[r]_{1\otimes\mu\otimes 1} \ar[dl]^{\mu\otimes\mu} & A^{\otimes 3}
\ar[dr]_{\mu(\mu\otimes 1)} & \\ A^{\otimes 2} \ar[rrr]_{\mu} & & & A } \end{proof}

Below we say a few words about a free braided monoidal category, containing  a commutative monoid. It has a nice
geometric presentation, where objects are points on an interval and morphisms are vines between them (see \cite{la,st}). We will denote free braided monoidal category, containing  a commutative monoid, by $\Vines$. An important property of the category $\Vines$ is its presentation in terms of $\B$ and $\Delta$. As in
the corollary \ref{dist}  the natural monoidal functors from $\B$ and $\Delta$ to the free braided monoidal category,
containing a commutative monoid, are isomorphic on objects.  Any morphism of the free braided monoidal category,
containing a monoid, can be decomposed into a morphism in $\B$ followed by a morphism in $\Delta$. The decomposition is
not unique. Two pairs $(\sigma\pi,\delta)$, $(\sigma,\delta)$ (here $\sigma, \pi$ are braids and $\delta$ is an order
preserving map) define the same morphism if and only if (the permutation associated with) $\pi$ stabilises the fibres
of $\delta$.

Since the unit map is a homomorphism of monoids, it follows from the lemma \ref{monm} that the semi-cosimplicial part
$(A^{\otimes *},\partial_*)$ (codegeneracies only) of the cosimplicial object associated with a monoid $A$ is a
semi-cosimplicial object in $\Mon(\cC)$ (all codegeneracies are homomorphisms of monoids). As a semi-cosimplicial
object it has the following property, which we call the covering condition.
\begin{defi}
\end{defi}
Let $M^*$ be a semi-cosimplicial monoid in a monoidal category $\cC$. For a collection $f_i:[n_i]\to [n], i=1,...,m,$
of injective order preserving maps such that $im(f_i)\cap im(f_j)=\emptyset$ for $i\not= j$ and $\cup_i im(f_i) = [n]$
define the {\em covering map} as the composition $$\otimes_{i=1}^m M^{n_i}\stackrel{\otimes_i M^{f_i}}{\longrightarrow}
(M^n)^{\otimes m}\stackrel{mult}{\longrightarrow} M^n$$ in $\cC$, where $mult$ is the iterated multiplication on $M^n$.
We say that $M^*$ satisfies the {\em covering condition} if any covering map is an isomorphism.
\medskip
By the lemma \ref{comm} a commutative monoid gives rise to a cosimplicial object $A^{\otimes *}$ in $\Common(\cC)$
satisfying the covering condition (all cosimplicial maps are homomorphisms of monoids).

\subsection{Monoids in symmetric monoidal categories}
\begin{lem}\label{comsym}
The subcategory $\Common(\cC)\subset\Mon(\cC)$ of commutative monoids in a symmetric monoidal category $\cC$ is
symmetric monoidal.
\end{lem}
\begin{proof}
It follows from the commutative diagram below that the tensor product of commutative monoids is commutative:
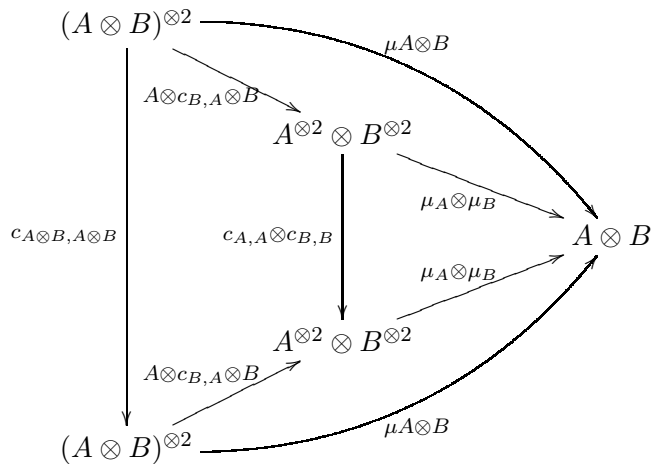
\xymatrix{ (A\otimes B)^{\otimes 2}\ar[rd]_{A\otimes c_{B,A}\otimes B} \ar @/^6ex/ [rrrdd]^{\mu{A\otimes B}}
\ar[dddd]_{c_{A\otimes B,A\otimes B}} & & &
\\
 & A^{\otimes 2}\otimes B^{\otimes 2} \ar[rrd]_{\mu_A\otimes\mu_B} \ar[dd]_{c_{A,A}\otimes c_{B,B}} & &
\\ & & & A\otimes B \\ & A^{\otimes 2}\otimes B^{\otimes 2} \ar[rru]^{\mu_A\otimes\mu_B} & & \\
(A\otimes B)^{\otimes 2}\ar[ru]^{A\otimes c_{B,A}\otimes B} \ar @/_6ex/ [rrruu]_{\mu{A\otimes B}} & & & }

Similarly the commutativity constraint $c_{A,A}:A\otimes A\to A\otimes A$ is a homomorphism of monoids:
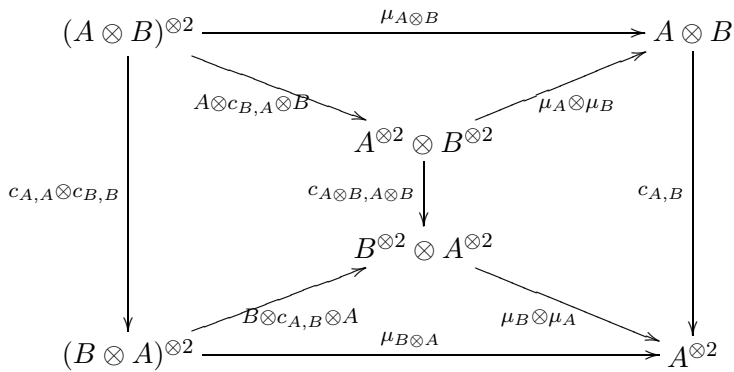
\xymatrix{(A\otimes B)^{\otimes 2} \ar[rrrr]^{\mu_{A\otimes B}} \ar[drr]_{A\otimes c_{B,A}\otimes B} \ar[ddd]_{c_{A,A}\otimes
c_{B,B}} & & & & A\otimes B\ar[ddd]_{c_{A,B}} \\ & & A^{\otimes 2}\otimes B^{\otimes 2} \ar[d]_{c_{A\otimes B,A\otimes B}}
\ar[urr]_{\mu_A\otimes\mu_B} & & \\ & & B^{\otimes 2}\otimes A^{\otimes 2} \ar[drr]_{\mu_B\otimes\mu_A} & & \\ (B\otimes A)^{\otimes 2}
\ar[rrrr]^{\mu_{B\otimes A}} \ar[urr]_{B\otimes c_{A,B}\otimes A} & & & & A^{\otimes 2}}

\end{proof}
\begin{rem}\label{sympow}
\end{rem}
It follows from the previous lemma that for any $n$ the symmetric group $S_n$ acts by monoid homomorphisms on the
tensor power $A^{\otimes n}$. Thus (when it exists) the subobject of invariants $S^nA$ (the joint equaliser of all
elements of $S_n$) is a commutative submonoid in $A^{\otimes n}$. We call it the {\em $n$-th symmetric power} of $A$.

We conclude this section by mentioning some properties of the cobar construction of a monoid in a symmetric monoidal
category as well as the well-known description of the free symmetric monoidal category, containing a commutative
monoid. Lemma \ref{bractcc} works obviously in the case of a cosimplicial object of a monoid in a symmetric monoidal
category. If the braiding is a symmetry, the braid group actions on the components of the cosimplicial object reduce to
the symmetric group actions. Moreover, by lemma \ref{comsym} this action is by monoid automorphisms, giving rise to a
{\em symmetric} cosimplicial monoid.

As in the braided case, the statement of lemma \ref{bractcc} can be interpreted as a distributive law. Now it
distributes order preserving maps over permutations $$\Delta(m,n)\times\bP(n,n)\to \bP(m,m)\times\Delta(m,n)$$ for.
Here $\bP(n,n)$ is the symmetric group on $n$ elements. This distributive law was used in \cite{ds} to characterise the
free symmetric monoidal category, containing a monoid, as a mixture $\bP\Delta$ of the free symmetric category on one
object $\bP$ and the free monoidal category containing a monoid $\Delta$. Note that the free symmetric monoidal
category, containing a monoid, has another description as the category of finite sets and maps with linear orders on
fibres (see \cite{ds} for the reference). The free symmetric monoidal category, containing a commutative monoid, also
has a simple combinatorial model.
\begin{prop}
The category of finite sets $\bS$ with monoidal structure given by coproduct is the free symmetric monoidal category,
containing a commutative monoid.
\end{prop}
\begin{proof}
The one-element set $[1]$ is a monoid in $\Sets$, with the unit map $[0]\to [1]$ and the multiplication given by the
epimorphism $[1]\otimes [1] = [2]\to [1]$. To show that $\Sets$ is freely generated by this monoid (as a symmetric
monoidal category) we need, for any commutative monoid $A$ in a symmetric (strict) monoidal category $\D$, to have a
symmetric monoidal functor $F:\Sets\to\D$, such that the monoid $F([1])$ is (isomorphic to) $A$. On objects (of the
skeletal model) of $\Sets$ the functor can be defined as follows $F([n]) = A^{\otimes n}$. To define its effect on
morphisms one can use the following factorisation property of morphisms of $\Sets$: any morphism can be decomposed
(non-uniquely) as a bijection followed by an order preserving map. To get such decomposition for a map $f:[m]\to [n]$
one need to fix linear orders on fibres of $f$, which gives a bijection $\sigma$ between $[m]$ and the ordered union
$\cup_{i\in [n]}[|f^{-1}(i)|]$ and an order preserving map $\delta:\cup_{i\in [n]}[|f^{-1}(i)|]\to [n]$. Now we can
define $F(f)$ as the composition of $F(\sigma):A^{\otimes m}\to A^{\otimes m}$ and $F(\delta):A^{\otimes m}\to
A^{\otimes n}$, where $F(\sigma)=\sigma$ is defined using the symmetric group action on $A^{\otimes m}$ and $F(\delta)$
is defined using the monoid structure on $A$ (as was explained in section \ref{mmc}). Note that the result is uniquely
defined (does not depend on the decomposition of $f$) since any too such decompositions differ by a fibre preserving
permutation $\pi$, and by commutativity of $A$, $\pi F(\delta) = F(\delta)$.
\end{proof}

\section{Braided and quasi-commutative monoids}
\subsection{Monoidal functors from the free braided category generated by a monoid}
Let $A$ be a monoid in a braided monoidal category $\cC$. Being morphisms in a braided category, $\iota$ and $\mu$ are
compatible with the braiding: $$c_{A,A}(\iota\otimes I) = I\otimes\iota,\ c_{A,A}(I\otimes\iota) = \iota\otimes I,$$
$$c_{A,A}(I\otimes\mu) = (\mu\otimes I)c_{A\otimes A,A},$$ $$c_{A,A}(\mu\otimes I) = (I\otimes\mu)c_{A,A\otimes A}.$$
Note that the axioms of braided monoidal category imply that $$c_{A\otimes A,A} = (c_{A,A}\otimes I)(I\otimes
c_{A,A}),\ c_{A,A\otimes A} = (I\otimes c_{A,A})(c_{A,A}\otimes I),$$ $$(c_{A,A}\otimes I)(I\otimes
c_{A,A})(c_{A,A}\otimes I) = (I\otimes c_{A,A})(c_{A,A}\otimes I)(I\otimes c_{A,A}).$$ Thus a monoidal functor
$F:\cC\to\D$ equips the monoid $F(A)$ with extra structure. The following definitions formalise this structure.
\begin{defi}
\end{defi}
Recall that a {\em Yang-Baxter operator} on an object $A$ is an isomorphism $R:A\otimes A\to A\otimes A$ satisfying
$$(R\otimes I)(I\otimes R)(R\otimes I) = (I\otimes R)(R\otimes I)(I\otimes R)$$ the so-called {\em Yang-Baxter
equation}. Let $A$ be a monoid in a monoidal category, with the multiplication $\mu:A\otimes A\to A$ and the unit map
$\iota:1\to A$. We call $A$ {\em braided} if there exists a Yang-Baxter operator $R$ on $A$ such that
\begin{equation}\label{ide}
R(\iota\otimes I) = I\otimes\iota,\ R(I\otimes\iota) = \iota\otimes I,
\end{equation}
\begin{equation}\label{mur}
R(I\otimes\mu) = (\mu\otimes I)(I\otimes R)(R\otimes I),
\end{equation}
\begin{equation}\label{mul}
R(\mu\otimes I) = (I\otimes\mu)(R\otimes I)(I\otimes R).
\end{equation}
We call $A$ {\em quasi-commutative} if in addition
\begin{equation}\label{com}
\mu R = \mu.
\end{equation}
A quasi-commutative monoid is {\em nearly commutative} if its Yang-Baxter operator satisfies the condition $R^2 = I$.

A {\em homomorphism of quasi-commutative monoids} $(A,R_A)\to(B,R_B)$ is a homomorphism of monoids $f:A\to B$ such
that $(f\otimes f)R_A = R_B(f\otimes f)$.

\begin{lem}\label{chqc}
Monoidal functors from the free braided monoidal category, containing a commutative monoid, into a monoidal category
$\D$ are in 1-1 correspondence with quasi-commutative monoids in $\D$.
\end{lem}
\begin{proof}

As was mentioned before the image under monoidal functor of a commutative monoid in a braided category is a
quasi-commutative monoid. In particular it gives a functor from the category of monoidal functors $\Mon(\Vines,\D)$
into the category of quasi-commutative monoids in $\D$. To prove that this is an equivalence we need to construct for a
quasi-commutative monoid $A\in\D$ a monoidal functor $F_A:\Vines\to\D$ which maps the generator $[1]$ of $\Vines$ into
$A$. On objects the functor is defined as follows $F_A([n]) = A^{\otimes n}$. To define it on morphisms one can use
factorisation property for morphisms of $\Vines$: $F_A(f) = F_A(\delta)F_A(\sigma)$ for a decomposition $f =
\delta\sigma$ into a braid $\sigma$ and an order preserving $\delta$. Here $F_A(\delta)$ is defined using the monoid
structure (see section \ref{mmc}) and $F_A(\sigma)$ is the image of the  homomorphism $B_n\to Aut(A^{\otimes n})$
associated with the Yang-Baxter operator $R$ (see for example \cite{ma}). Commutativity of the monoid $A$ implies that
the result $F_A(f)$ is well-defined (does not depend on the decomposition). Indeed, two decompositions differ by a
braid $\pi$, stabilising the fibres of $\delta$. Now the commutativity of $A$ implies that $F_A(\delta)F_A(\pi) =
F_A(\delta)$.
\end{proof}

\subsection{Main results}

The following theorem gives a characterisation of quasi-commutative monoids in purely algebraic terms.
\begin{theo}\label{main}
For an object $A$ in a monoidal category the following data are equivalent:
\newline
i) a structure of quasi-commutative monoid on $A$,
\newline
ii) a cosimplicial complex of monoids $A^*$ with $A^1=A$, satisfying the covering condition,
\newline
iii) a length 3 truncated cosimplicial complex of monoids $A^*$ with $A^1=A$, satisfying the covering condition.
\end{theo}
\begin{proof}
The implication i)$\Rightarrow$ ii) is provided by the lemma \ref{chqc}. Indeed, a quasi-commutative monoid $A$ defines
a monoidal functor from the free monoidal category, containing a commutative monoid. The image of this functor is a
braided monoidal subcategory in which $A$ is a commutative monoid. Thus, according to section \ref{mbc}, we can form
the cosimplicial complex $A^{\otimes *}$ which will satisfy the covering condition.

The implication ii)$\Rightarrow$ iii) is obvious.

The least trivial part of the proof is the implication iii)$\Rightarrow$ i). For a truncated cosimplicial complex of
monoids $A^*$ of length 3, satisfying the covering condition, we will show that $A^1$ is quasi-commutative with respect
to the Yang-Baxter operator $R = (\mu(\partial_0\otimes\partial_1))^{-1}\mu(\partial_1\otimes\partial_0)$. We check the
defining equations for $(A^1,R)$ diagrammatically using the cosimplicial identities. Condition (\ref{com}) follows from
the diagram: $$\begin{xy}\xymatrix{ (A^1)^{\otimes 2} \ar[rr]^{\partial_1\otimes\partial_0} \ar[dd]_R \ar[dr]_\mu & &
(A^2)^{\otimes 2} \ar[dl]^{\mu\otimes\sigma\otimes\sigma} \ar[dr]^\mu \\ & A^1 & & A^2 \ar[ll]_\sigma \\ (A^1)^{\otimes
2} \ar[ru]^\mu \ar[rr]^{\partial_0\otimes\partial_1} & & (A^2)^{\otimes 2} \ar[ul]_{\mu\otimes\sigma\otimes\sigma}
\ar[ru]_\mu }\end{xy}$$

Conditions (\ref{ide}) are guaranteed by the diagrams: $$\begin{xy} (-2,0)*+!R{\xybox{\xymatrix{ (A^1)^{\otimes 2}
\ar[rr]^{\partial_1\otimes\partial_0} \ar[dd]_R & & (A^2)^{\otimes 2} \ar[dr]^\mu \\ & A^1 \ar[ul]^{\iota\otimes I}
\ar[ru]_{\iota\otimes\partial_0} \ar[rr]^{\partial_0} \ar[dr]^{\partial_0\otimes\iota} \ar[dl]_{I\otimes\iota} & & A^2
\\ (A^1)^{\otimes 2} \ar[rr]^{\partial_0\otimes\partial_1} & & (A^2)^{\otimes 2} \ar[ur]_\mu }}}
\POS(2,0)*+!L{\xybox{\xymatrix{ (A^1)^{\otimes 2} \ar[rr]^{\partial_1\otimes\partial_0} \ar[dd]_R & &
(A^2)^{\otimes 2} \ar[dr]^\mu \\ & A^1 \ar[ul]^{I\otimes\iota} \ar[ru]_{\partial_1\otimes\iota}
\ar[rr]^{\partial_1} \ar[dr]^{\partial_1\otimes\iota} \ar[dl]_{\iota\otimes I} & & A^2 \\ (A^1)^{\otimes 2}
\ar[rr]^{\partial_0\otimes\partial_1} & & (A^2)^{\otimes 2} \ar[ur]_\mu }}}\end{xy}$$

For condition (\ref{mur}) we have the following:

\xymatrix{ (A^1)^{\otimes 2} \ar[dddddddddd]_R \ar[rrrr]^{\partial_1\otimes\partial_0} & & & & (A^2)^{\otimes 2}
\ar[rddddd]^\mu \\ & (A^1)^{\otimes 3} \ar[dddd]_{R\otimes I} \ar[ul]^{I\otimes\mu}
\ar[dr]_{\partial_1\otimes\partial_0\otimes I} \ar[rr]^{\partial_1\otimes\partial_0\otimes\partial_0} & &
(A^2)^{\otimes 3} \ar[ur]^{I\otimes\mu} \ar[dr]^{\mu\otimes I} & & \\ & & (A^2)^{\otimes 2}\otimes A^1
\ar[ur]_{I\otimes I\otimes\partial_0} \ar[dr]_{\mu\otimes I} & & (A^2)^{\otimes 2} \ar[rddd]_\mu & \\ & & &
A^2\otimes A^1 \ar[ur]^{I\otimes\partial_0} & & \\ & & (A^2)^{\otimes 2}\otimes A^1 \ar[ur]^{\mu\otimes I}
\ar[dr]^{I\otimes I\otimes\partial_0} & & & \\ & (A^1)^{\otimes 3} \ar[ur]^{\partial_0\otimes\partial_1\otimes I}
\ar[dddd]_{I\otimes R} \ar[dr]_{I\otimes\partial_1\otimes\partial_0}
\ar[rr]^{\partial_0\otimes\partial_1\otimes\partial_0} & & (A^2)^{\otimes 3} \ar[ruuu]_{\mu\otimes I}
\ar[rddd]^{I\otimes\mu} & & A^2 \\ & & A^1\otimes (A^2)^{\otimes 2} \ar[dr]_{I\otimes\mu}
\ar[ur]_{\partial_0\otimes I\otimes I} & & & \\ & & & A^1\otimes A^2 \ar[dr]_{\partial_0\otimes I} & & \\ & &
A^1\otimes (A^2)^{\otimes 2} \ar[dr]^{\partial_0\otimes I\otimes I} \ar[ur]^{I\otimes\mu} & & (A^2)^{\otimes 2}
\ar[ruuu]^\mu & \\ & (A^1)^{\otimes 3} \ar[ru]^{I\otimes\partial_0\otimes\partial_1} \ar[dl]^{\mu\otimes I}
\ar[rr]_{\partial_0\otimes\partial_0\otimes\partial_1} & & (A^2)^{\otimes 3} \ar[ur]_{I\otimes\mu}
\ar[dr]_{\mu\otimes I} & & \\ (A^1)^{\otimes 2} \ar[rrrr]_{\partial_0\otimes\partial_1} & & & & (A^2)^{\otimes 2}
\ar[ruuuuu]_\mu }

\medskip
Analogously condition (\ref{mur}) follows from:

\xymatrix{ (A^1)^{\otimes 2} \ar[dddddddddd]_R \ar[rrrr]^{\partial_1\otimes\partial_0} & & & & (A^2)^{\otimes 2}
\ar[rddddd]^\mu \\ & (A^1)^{\otimes 3} \ar[dddd]_{I\otimes R} \ar[ul]^{\mu\otimes I}
\ar[dr]_{I\otimes\partial_1\otimes\partial_0} \ar[rr]^{\partial_1\otimes\partial_1\otimes\partial_0} & &
(A^2)^{\otimes 3} \ar[dr]^{I\otimes\mu} \ar[ur]^{\mu\otimes I} & & \\ & & A^1\otimes (A^2)^{\otimes 2}
\ar[ur]_{\partial_1\otimes I\otimes I} \ar[dr]_{I\otimes\mu} & & (A^2)^{\otimes 2} \ar[rddd]_\mu & \\ & & &
A^2\otimes A^1 \ar[ur]^{\partial_1\otimes I} & & \\ & & (A^2)^{\otimes 2}\otimes A^1 \ar[ur]^{I\otimes\mu}
\ar[dr]^{\partial_1\otimes I\otimes I} & & & \\ & (A^1)^{\otimes 3} \ar[ur]^{I\otimes\partial_0\otimes\partial_1}
\ar[dddd]_{R\otimes I} \ar[dr]_{\partial_1\otimes\partial_0\otimes I}
\ar[rr]^{\partial_1\otimes\partial_0\otimes\partial_1} & & (A^2)^{\otimes 3} \ar[rddd]^{\mu\otimes I}
\ar[ruuu]_{I\otimes\mu} & & A^2 \\ & & (A^2)^{\otimes 2}\otimes A^1 \ar[dr]_{\mu\otimes I} \ar[ur]_{I\otimes
I\otimes\partial_1} & & & \\ & & & A^2\otimes A^1 \ar[dr]_{I\otimes\partial_1} & &
\\ & & (A^2)^{\otimes 2}\otimes A^1 \ar[dr]^{I\otimes I\otimes\partial_1} \ar[ur]^{\mu\otimes I} & & (A^2)^{\otimes 2}
\ar[ruuu]^\mu & \\ & (A^1)^{\otimes 3} \ar[ru]^{\partial_0\otimes\partial_1\otimes I} \ar[dl]^{I\otimes\mu}
\ar[rr]_{\partial_0\otimes\partial_1\otimes\partial_1} & & (A^2)^{\otimes 3} \ar[ur]_{\mu\otimes I}
\ar[dr]_{I\otimes\mu} & & \\ (A^1)^{\otimes 2} \ar[rrrr]^{\partial_0\otimes\partial_1} & & & & (A^2)^{\otimes 2}
\ar[ruuuuu]_\mu }

\medskip
Finally for the Yang-Baxter equation we have the following:

$$ \xygraph{ !{0;/r3.3pc/:;/u6.5pc/::}[]*+{(A^1)^{\otimes 3}} (
  ( :[dl]*+{A^1\otimes(A^2)^{\otimes 2}}_{I\otimes\partial_1\otimes\partial_0}
    ( :[dr]*+{(A^3)^{\otimes 3}}="0,2" _{\partial_2\partial_1\otimes\partial_0\otimes\partial_0}
      ( :[dl]*+{(A^3)^{\otimes 2}}="-1,1" ^{I\otimes\mu}
      ,
        :[dr]*+{(A^3)^{\otimes 2}}="1,1" ^{\mu\otimes I}
      )
    ,
      :[dl]*+{A^1\otimes A^2}="-2,2" _{I\otimes\mu}
      :"-1,1" ^{\partial_2\partial_1\otimes\partial_0}
      :[dr]*+{A^3}="0,0" ^{\mu}
    )
  ,
    :[dr]*+{(A^2)^{\otimes 2}\otimes A^1} ^{\partial_1\otimes\partial_0\otimes I}
    ( :[dr]*+{A^2\otimes A^1}="2,2" ^{\mu\otimes I}
      :"1,1" ^{\partial_2\otimes\partial_1\partial_0}
      : "0,0" ^{\mu}
    ,
      :"0,2" ^{\partial_2\otimes\partial_2\otimes\partial_1\partial_0}
    )
  )
  :[d(2)l(6)]*+{(A^1)^{\otimes 3}} _{I\otimes R}
  ( :[rr]*+{A^1\otimes(A^2)^{\otimes 2}} ^{I\otimes\partial_0\otimes\partial_1}
    ( :"-2,2" ^{I\otimes\mu}
    ,
      :[dr]*+{(A^3)^{\otimes 3}}="-3,1" ^{\partial_2\partial_1\otimes\partial_0\otimes\partial_0}
      ( :"-1,1" ^{I\otimes\mu}
      ,
        :[dr]*+{(A^3)^{\otimes 2}}="-2,0" ^{\mu\otimes I}
        :"0,0" ^\mu
      )
    )
  ,
    :[dr]*+{(A^2)^{\otimes 2}\otimes A^1} ^{\partial_1\otimes\partial_0\otimes I}
    ( :"-3,1" ^{\partial_1\otimes\partial_1\otimes\partial_2\partial_0}
    ,
      :[dr]*+{A^2\otimes A^1}="-4,0" ^{\mu\otimes I}
      :"-2,0" ^{\partial_1\otimes\partial_2\partial_0}
    )
  )
  :[d(4)]*+{(A^1)^{\otimes 3}} _{R\otimes I}
  ( :[ur]*+{(A^2)^{\otimes 2}\otimes A^1} _{\partial_0\otimes\partial_1\otimes I}
    ( :"-4,0" _{\mu\otimes I}
    ,
      :[rr]*+{(A^3)^{\otimes 3}}="-3,-1" ^{\partial_1\otimes\partial_1\otimes\partial_2\partial_0}
      ( :"-2,0" _{\mu\otimes I}
      ,
        :[rr]*+{(A^3)^{\otimes 2}}="-1,-1" ^{I\otimes\mu}
        :"0,0" _\mu
      )
    )
  ,
    :[rr]*+{A^1\otimes(A^2)^{\otimes 2}} ^{I\otimes\partial_1\otimes\partial_0}
    ( :"-3,-1" _{\partial_0\partial_0\otimes\partial_2\otimes\partial_2}
    ,
      :[rr]*+{A^1\otimes A^2}="-2,-2" ^{I\otimes\mu}
      :"-1,-1" _{\partial_0\partial_0\otimes\partial_2}
    )
  )
  :[d(2)r(6)]*+{(A^1)^{\otimes 3}}="0,-4" _{I\otimes R}
  ( :[ul]*+{A^1\otimes(A^2)^{\otimes 2}} ^{I\otimes\partial_0\otimes\partial_1}
    ( :"-2,-2" ^{I\otimes\mu}
    ,
      :[ur]*+{(A^3)^{\otimes 3}}="0,-2" ^{\partial_0\partial_0\otimes\partial_2\otimes\partial_2}
      ( :"-1,-1" _{I\otimes\mu}
      ,
        :[ur]*+{(A^3)^{\otimes 2}}="1,-1" _{\mu\otimes I}
        :"0,0" _\mu
      )
    )
  ,
    :[ur]*+{(A^2)^{\otimes 2}\otimes A^1} _{\partial_0\otimes\partial_1\otimes I}
    ( :"0,-2" _{\partial_0\otimes\partial_0\otimes\partial_2\partial_1}
    ,
      :[ur]*+{A^2\otimes A^1}="2,-2" _{\mu\otimes I}
      :"1,-1" _{\partial_0\otimes\partial_2\partial_1}
    )
  )
,
  :[d(2)r(6)]*+{(A^1)^{\otimes 3}} ^{R\otimes I}
  ( :[ll]*+{(A^2)^{\otimes 2}\otimes A^1} _{\partial_0\otimes\partial_1\otimes I}
    ( :"2,2" _{\mu\otimes I}
    ,
      :[dl]*+{(A^3)^{\otimes 3}}="3,1" ^{\partial_2\otimes\partial_2\otimes\partial_1\partial_0}
      ( :"1,1" _{\mu\otimes I}
      ,
        :[dl]*+{(A^3)^{\otimes 2}}="2,0" ^{I\otimes\mu}
        :"0,0" _\mu
      )
    )
  ,
    :[dl]*+{A^1\otimes(A^2)^{\otimes 2}} ^{I\otimes\partial_1\otimes\partial_0}
    ( :"3,1" _{\partial_0\partial_1\otimes\partial_1\otimes\partial_1}
    ,
      :[dl]*+{A^1\otimes A^2}="4,0" ^{I\otimes\mu}
      :"2,0" _{\partial_0\partial_1\otimes\partial_1}
    )
  )
  :[d(4)]*+{(A^1)^{\otimes 3}} ^{I\otimes R}
  ( :[ul]*+{A^1\otimes(A^2)^{\otimes 2}} _{I\otimes\partial_0\otimes\partial_1}
    ( :"4,0" _{I\otimes\mu}
    ,
      :[ll]*+{(A^3)^{\otimes 3}}="3,-1" _{\partial_0\partial_1\otimes\partial_1\otimes\partial_1}
      ( :"2,0" _{I\otimes\mu}
      ,
        :"1,-1" _{\mu\otimes I}
      )
    )
  ,
    :[ll]*+{(A^2)^{\otimes 2}\otimes A^1} _{\partial_0\otimes\partial_1\otimes I}
    ( :"3,-1" _{\partial_0\otimes\partial_0\otimes\partial_2\partial_1}
    ,
      :"2,-2" _{\mu\otimes I}
    )
  )
  : "0,-4" ^{R\otimes I}
)
}$$ \end{proof}

The previous characterisation of quasi-commutative monoids in terms of (truncated) cosimplicial objects can be easily
generalisied to braided monoids and nearly commutative monoids.
\begin{theo}\label{main1}
A structure of a braided monoid on an object $A$ of a monoidal category $\cC$ is equivalent to
\newline
i) a semi-cosimplicial complex of monoids $A^*$ in $\cC$ with $A^1=A$, satisfying the covering condition,
\newline
ii) a length 3 truncated semi-cosimplicial complex of monoids $A^*$ with $A^1=A$, satisfying the covering condition.
\end{theo}
\begin{proof}
The only place in the proof where we used a degeneracy map was in the verification of the commutativity.
\end{proof}

\begin{theo}\label{main2}
A structure of an nearly commutative monoid on object $A$ of a monoidal category $\cC$ is equivalent to
\newline
i) a symmetric cosimplicial complex of monoids $A^*$ in $\cC$ with $A^1=A$, satisfying the covering condition,
\newline
ii) a length 3 truncated symmetric cosimplicial complex of monoids $A^*$ with $A^1=A$, satisfying the covering
condition.
\end{theo}

\section{Some examples}
Below we list some examples of quasi-commutative monoids in different monoidal categories (mostly of algebraic origin).
In some cases it is easier to give the quasi-commutative structure in a straightforward way by presenting the
Yang-Baxter operator and checking the identities; in other cases it is much easier to define the cosimplicial complex
and check the covering condition. Finally relations (monoidal functors) between ambient monoidal categories allow us to
turn one series of examples into another, or sometimes produce new examples.

\begin{exam}\label{gr} Groups are quasi-commutative.
\end{exam}
Here the category is the category of sets $\Sets$ with monoidal structure given by cartesian product. A monoid in this
category is just a monoid. In particular any group is a monoid in $\Sets$. For a group $G$ define a cosimplicial
complex $G^*$ with $G^n = G^{\times n}$ the cartesian product and codegeneration and coface maps: $$\sigma_i:G^{\times
n}\to G^{\times n-1},\quad
\partial_i:G^{\times n}\to G^{\times n+1},$$ where $$\sigma_i(x_1,...,x_n) =
(x_1,...,x_{i+1},x_{i+3},...,x_n),\ i=0,...,n-2,$$ $$\partial_i(x_1,...,x_n) =
(x_1,...,x_i,x_{i+1},x_{i+1},x_{i+2},...,x_n),\ i=0,...,n-1,$$ $$\partial_n(x_1,...,x_n) = (x_1,...,x_n,e).$$ Note that
the above formulas do not involve inverses thus defining a cosimplicial structure on $G^{\times *}$ for any monoid $G$.
But the covering condition is fulfilled only when $G$ is a group. Indeed, the bijectivity of the composition $$G\times
G\stackrel{\partial_0\times\partial_1}{\longrightarrow} G^{\times 4}\stackrel{\mu_{G^{\times 2}}}{\longrightarrow}
G^{\times 2}$$ implies that for any $a,b\in G$ the system $$xy = a,\ x = b$$ has unique solution which means that $b$
is invertible. In particular, the inverse to the map $\mu_{G^{\times 2}}(\partial_0\otimes\partial_1)$ has the form
$(x,y)\mapsto (y,y^{-1}x)$ and the Yang-Baxter operator is given by
\begin{equation}\label{gyb}
R:G\times G\to G\times G,\quad R(x,y) = (y,y^{-1}xy).
\end{equation}
Note that
\begin{equation}\label{r2}
R^2(x,y) = R(y,y^{-1}xy) = (y^{-1}xy,(y^{-1}xy)^{-1}y(y^{-1}xy)) = (y^{-1}xy,y^{-1}x^{-1}yxy)
\end{equation}
so that a group $G$ with the Yang-Baxter operator $R$ is nearly commutative if and only if $G$ is abelian, in which
case $R$ is just the symmetry in $\Sets$.

The functor $\Sets\to \Vect$, sending a set $X$ into the vector space $k[X]$ spanned by it, is monoidal. Thus for any
group the groups algebra $k[G]$ is quasi-commutative in $\Vect$ with respect to the Yang-Baxter operator $R$.

\begin{exam}\label{ex} Central extensions are quasi-commutative.
\end{exam}
This is a modification of the previous example. Let $A$ be an abelian group and $\Sets_f^A$ is the category of faithful
$A$-sets. Define the monoidal product of $A$-sets $X$ and $Y$ to be $X\times_A Y = (X\times Y)/A$ the quotient set of
the cartesian product by the anti-diagonal action of $A$. In another words $X\times_A Y$ is the set of pairs $(x,y)$
modulo relations $(ax,y) = (x,ay)$ with obvious $A$-action $a(x,y) = (ax,y) = (x,ay)$. A monoid in the category
$\Sets_f^A$ with the tensor product $\times_A$ is a monoid (in $\Sets$) together with a central inclusion $A\to M$. In
particular any central extension of groups $A\subset G$ is a monoid in $\Sets_f^A$. Note that for a central extension
of groups $A\to G$ the Yang-Baxter operator (\ref{gyb}) on the group $G$ preserves the relations $(ax,y) = (x,ay)$ and
the $A$-action on $G\times_A G$ thus making the extension $A\subset G$ a quasi-commutative monoid in $\Sets_f^A$.

An extension $A\to G$ is nearly commutative iff the quotient group $G/A$ is abelian. Indeed, by (\ref{r2}) $R^2$ is the
identity on $G\times_A G$ iff for any $x,y\in G$ $(y^{-1}xy,y^{-1}x^{-1}yxy) = (a^{-1}x,ay)$ for some $a\in A$. Solving
this system we can rewrite the condition as $[x,y] \in A$, where $[x,y] = xyx^{-1}y^{-1}$ is the commutator. So if the
extension $A\to G$ is nearly commutative $[G,G]\subset A$ and $G/A$ becomes abelian. Conversely, if $G/A$ is abelian
then $[x,y] = [x,xy]\in A$ for any $x,y\in G$.

Fix a homomorphism $\chi:A\to k^*$ into the invertible elements of the field $k$. Define the $\chi$-span $k_\chi[X]$ of
a (faithful) $A$-set as the quotient of $k[X]$ modulo the relations $a(x) = \chi(a)x$ for $a\in A$, $x\in X$. This
construction is clearly functorial, the functor $k_\chi[\ ]:\Sets_f^A\to\Vect_k$ is monoidal (transforming product
$\times_A$ of $A$-sets into tensor product of vector spaces). Again monoids in $\Sets_f^A$ give rise to algebras over
$k$ and quasi-commutative monoids correspond to quasi-commutative algebras. In particular according to the example
(\ref{ex}) for any central extensions of groups $A\to G$ the algebra $k_\chi[G]$ is quasi-commutative. Note that the
algebra $k_\chi[G]$ is a {\em skew group algebra} $k[G/A,\alpha]$, where the 2-cocycle $\alpha\in Z^2(G/A,k^*)$ is
$\chi(\gamma)$ with $\gamma\in Z^2(G/A,A)$ being a cocycle of the central extension $A\to G\to G/A$. Recall that for a
2-cocyle $\alpha\in Z^2(S,k^*)$ the skew group algebra $k[S,\alpha]$ is the vector space over $k$ spanned by $e_s$ for
$s\in S$ with the product $$e_se_t = \alpha(s,t)e_{st}.$$ Note that the 2-cocycle condition is equivalent to the
associativity of this multiplication and the isomorphism class of a skew group algebra depends only on the cohomology
class of the 2-cocycle. Our construction establishes quasi-commutativity of a skew group algebra $k[S,\alpha]$ with the
Yang-Baxter operator:
\begin{equation}\label{ceyb}
R(e_s\otimes e_t) = \frac{\alpha(s,t)}{\alpha(t,t^{-1}st)}(e_t\otimes e_{t^{-1}st}).
\end{equation}

\begin{exam}Hopf algebras are quasi-commutative.
\end{exam}
Let $H$ be a Hopf algebra with a unit map $\iota$, coproduct $\triangle$, counit $\varepsilon$ and an invertible
antipode $S$. Define a cosimplicial complex $H^*$ with $H^n = H^{\otimes n}$ and the coface $\sigma_i:H^{\otimes n}\to
H^{\otimes n-1},\ i=0,...,n-2$ and codegeneration maps $\partial_j:H^{\otimes n}\to H^{\otimes n+1},\ j=0,...,n$:
$$\sigma_i = I_{i+1}\otimes\varepsilon\otimes I_{n-i-2},\quad \partial_j = \left\{\begin{array}{ll}
I_j\otimes\triangle\otimes I_{n-j-1}, & j<n+1\\ I_n\otimes\iota, & j=n+1 \end{array}\right.$$

The maps $\mu_{H^{\otimes 2}}(\partial_0\otimes\partial_1),\mu_{H^{\otimes 2}}(\partial_0\otimes\partial_1):H^{\otimes
2}\to H^{\otimes 2}$ have the form $$g\otimes h\mapsto \triangle(g)(h\otimes 1),\quad g\otimes h\mapsto (g\otimes
1)\triangle(h)$$ correspondingly. Their invertibility follows from the invertibility of the antipode. For example, the
inverse of the first map is given by $$g\otimes h\mapsto (I\otimes S^{-1})t\triangle(h)(1\otimes g),$$ where $t$ is the
transposition of tensor factor and $t\triangle$ is the opposite coproduct.

Thus any Hopf algebra $H$ with invertibe antipode is quasi-commutative with respect to the Yang-Baxter operator
\begin{equation}\label{hyb}
R(g\otimes h) = \sum_{(h)}h_{(2)}\otimes S^{-1}(h_{(1)})gh_{(0)}.
\end{equation}
Here we use so-called Sweedler's notation (see \cite{sw}) according to which $$\triangle(f) = \sum_{(f)}f_{(0)}\otimes
f_{(1)},\quad (\triangle\otimes I)\triangle(f) = (I\otimes\triangle)\triangle(f) = \sum_{(f)}f_{(0)}\otimes
f_{(1)}\otimes f_{(2)},...$$ Similarly to the group case, the Yang-Baxter operator is involutive (i.e. the structure is
nearly-commmutative) if and only if the Hopf algebra $H$ is commutative. Note that in that case the Yang-Baxter
operator is the ordinary transposition.

\begin{exam} Galois algebras are quasi-commutative.
\end{exam}
Let $H$ be a Hopf algebra and $A$ be a (right) $H$-comodule algebra with the coaction $\psi:A\to A\otimes H$ (see
\cite{mon}). Define a cosimplicial complex with $n$-th term $A\otimes H^{\otimes n-1}$ and the coface
$\sigma_i:A\otimes H^{\otimes n-1}\to A\otimes H^{\otimes n-2},\ i=0,...,n-2$ and codegeneration maps
$\partial_j:A\otimes H^{\otimes n-1}\to A\otimes H^{\otimes n},\ j=0,...,n$: $$\sigma_i =
I_{i+1}\otimes\varepsilon\otimes I_{n-i-2},\quad \partial_j = \left\{ \begin{array}{ll} \psi\otimes I_{n-1}, & j=0\\
I_j\otimes\triangle\otimes I_{n-j-1}, & 0<j<n+1\\ I_n\otimes\iota, & j=n+1 \end{array}\right. $$ The maps
$\mu_{A\otimes H}(\partial_0\otimes\partial_1),\mu_{A\otimes H}(\partial_0\otimes\partial_1):H^{\otimes 2}\to
H^{\otimes 2}$ have the form $$a\otimes b\mapsto \psi(a)(b\otimes 1),\quad a\otimes b\mapsto (a\otimes 1)\psi(h)$$
correspondingly. Their invertibility is equivalent to the {\em Galois} property of the coaction \cite{mon}. Since it is
rather hard to invert these maps, the general form of the corresponding Yang-Baxter operator is out of reach. However,
it is much easier to find an operator $\tau:A\otimes H\to A\otimes H$ satisfying $$\tau\partial_0 = \partial_1,\quad
\tau\partial_1 = \partial_0,\quad \sigma\tau = \sigma.$$ Indeed, $\tau(a\otimes h) = \psi(a)(1\otimes S(h))$ solves
these equations. The explicit form of $\tau$ allows us to see when the quasi-commutative structure on $A$ is
nearly-commutative. It is straightforward that $\tau^2 = I$ if and only if the Hopf algebra $H$ is commutative. In
contrast to the previous example the Yang-Baxter operator corresponding to a non-trivial Galois algebra over
commutative Hopf algebra can be non-trivial.

For instance, let $H=k[G]$ be the group algebra of a group $G$. It is well-known (see \cite{mon}) that a
$k[G]$-comodule Galois algebra is a skew group algebra $k[G,\alpha]$ with coaction given by $\psi(e_g) = e_g\otimes g$.
The maps $\mu_{A\otimes H}(\partial_0\otimes\partial_1),\mu_{A\otimes H}(\partial_0\otimes\partial_1):H^{\otimes 2}\to
H^{\otimes 2}$ now take the form $$e_f\otimes e_g\mapsto e_fe_g\otimes f = \alpha(f,g)e_{fg}\otimes f,\quad e_f\otimes
e_g\mapsto e_fe_g\otimes g = \alpha(f,g)e_{fg}\otimes g.$$ The inverse of the first map can be given explicitly
$$e_f\otimes g\mapsto \alpha(g,g^{-1}f)e_g\otimes e_{g^{-1}f}.$$ Thus we recover the Yang-Baxter operator (\ref{ceyb})
on $k[G,\alpha]$. When $G$ is abelian the Yang-Baxter operator (\ref{ceyb}) reduces to $$R(e_f\otimes e_g) =
\frac{\alpha(f,g)}{\alpha(g,f)}e_g\otimes e_f,$$ which is nearly commutative. Note that, for a 2-cocyle $\alpha\in
Z^2(G,k^*)$ of an abelian $G$, the expression $\alpha(f,g)\alpha(g,f)^{-1}$ defines a skew-symmetric bi-multiplicative
form on $G$.

\section{Some applications}
Here we use theorems \ref{main},\ref{main1},\ref{main2} to describe braided, quasi- and nearly commutative structures
on groups in terms of certain group theoretic data. Note that the resulting Yang-Baxter operators are known (see
\cite{ess,lyz}).

A {\em matched pair} of groups $F,H$ is a group $G$, which can be written as a product $FH$. In that case the
multiplication in $G$ defines two functions $\alpha:H\times F\to F,\ \beta:H\times F\to H$ such that $hf =
\alpha(h,f)\beta(h,f)$.

A {\em 1-cocycle} of the group $G$ with coefficients in the group $K$ (on which $G$ acts from the right by group
automorphisms) is a map $\phi:G\to K$ such that $$\phi(fg) = \phi(f)^g\phi(g),\ \forall f,g\in G.$$ With a pair of
groups $G,K$ where the first acts on the second by group automorphisms we can associate their {\em semi-direct product}
$G\ltimes K$ which set-theoretically is the product of $G$ and $K$ with the multiplication given by: $$(f,u)(g,v) =
(fg,u^g v),\quad xf,g\in G, u,v\in K.$$
\begin{theo}
Braided structures on a group $G$ correspond to matched pairs of $G$ with itself. The Yang-Baxter operator, defined by
a matched pair $\alpha,\beta:G\times G\to G$ has the form:
\begin{equation}\label{ybmp}
R(f,g) = (\alpha(f,g),\beta(f,g)).
\end{equation}

A quasi-commutative structure on a group $G$ corresponds to a group $K$ with $G$ acting on it by automorphisms and a
1-cocycle $\phi:G\to K$ bijective as a set-theoretic map so that the Yang-Baxter operator has the form: $$R(f,g) =
(fg\psi(f,g)^{-1},\psi(f,g)),\quad\mbox{where}\ \psi(f,g) = \phi^{-1}(\phi(f)^{g}).$$ The Yang-Baxter operator is
involutive (the structure is nearly commutative) iff the group $K$ is abelian.
\end{theo}
\begin{proof}
By theorem \ref{main} braided structures on a group $G$ correspond to (length 3 truncated) semi-cosimplicial complexes
of monoids $G^*$ with $G=G^1$ satisfying the covering condition. First we show that for such a complex all $G^n$ are
groups. Indeed, by the covering condition any element of $G^n$ can be written as a product of homomorphic images of
elements from $G^1$. Since $G^1$ is a group, homomorphic images of its elements are invertible and so are their
products. Now the covering condition implies that any element of $G^2$ can be uniquely written as a product of elements
of $\partial_0(G)$ and $\partial_1(G)$. The Yang-Baxter operator $R$ on $G^1$ is determined by the relation
$\mu_{G^2}(\partial_0\otimes\partial_1)R = \mu_{G^2}(\partial_1\otimes\partial_0)$, which implies the formula
(\ref{ybmp}).

Now by theorem \ref{main} quasi-commutative structures on a group $G$ correspond to cosimplicial complexes of monoids
$G^*$ with $G=G^1$ satisfying the covering condition. The homomorphisms $\partial_0:G^1\to G^2$, $\sigma_0:G^2\to G^1$
identifies $G^2$ with a semi-direct product $G\ltimes K$ where $K=ker(\sigma_0)$. So that implies $$\sigma_0(f,u) =
f,\quad
\partial_0(f) = (f,e),\ f\in G^1, u\in K.$$ Since $\partial_1:G^1\to G^2$ is also split by $\sigma_0:G^2\to G^1$ it
can be written as $\partial_1(f) = (f,\phi(f))$ for some map $\phi:G^1\to K$. More explicitly, $\phi(f) =
\partial_0(f)^{-1}\partial_1(f)$. Multiplicativity of $\partial_1$ implies that $\phi$ is 1-cocycle: $$\phi(fg) =
\partial_0(g)^{-1}\partial_0(f)^{-1}\partial_1(f)\partial_1(g) =
\partial_0(g)^{-1}\phi(f)\partial_0(g)\partial_0(f)^{-1}\partial_1(g) = \phi(f)^g\phi(g).$$
\newline
As before we can get the Yang-Baxter operator $R$ on $G^1$ using the relation $\mu_{G^2}(\partial_0\otimes\partial_1)R
= \mu_{G^2}(\partial_1\otimes\partial_0)$. For $R(f,g) = (z,w)$ $$\mu_{G^2}(\partial_1\otimes\partial_0)(f,g) =
(f,\phi(f))(g,1) = (fg,\phi(f)^g)$$ coincides $$\mu_{G^2}(\partial_0\otimes\partial_1)(z,w) = (z,e)(w,\phi(w)) =
(zw,\phi(w))$$ which imply that $z = fgw^{-1}, w = \phi^{-1}(\phi(f)^g)$.

It turns out that in the group case a length 3 cosimplicial complex with the covering condition is determined by its
length 2 part, so to check the Yang-Baxter equation for $R$ is enough to construct maps $\partial_i:G^2\to G^3$,
$\sigma_j:G^3\to G^2$ for $i=0,1,2, j=0,1$ satisfying the cosimplicial identities and the covering condition. Define
$G^3 = G\ltimes(K\times K)$ where the semi-direct product is taken with respect to the diagonal action of $G$ on $K$.
Define $\partial_i$ and $\sigma_j$ by $$\begin{array}{ll} \partial_0(f,u) = (f,u,e), & \sigma_0(f,u,v) = (f,u), \\
\partial_1(f,u) = (f,u,u), & \sigma_1(f,u,v) = (f,v), \\ \partial_2(f,u) = (f,\phi(f),u), &
\end{array}$$ The cosimplicial identities and the covering condition can be checked directly.

Finally in the nearly commutative case the automorphism $\tau:G^2\to G^2$ defined by
\begin{equation}\label{tau}
\tau\mu(\partial_0\otimes\partial_1) = \mu(\partial_1\otimes\partial_0)
\end{equation}
is a group automorphism. Applying both sides to $(f,g)\in G^{\times 2}$ we get $\tau(fg,\phi(g)) = (fg,\phi(f)^g)$.
Thus for $(z,u)\in G\ltimes K = G^2$ $\tau(z,u) = (z,\phi(z)u^{-1})$. Since $\phi$ is bijective, $\tau$ is a
homomorphism iff $K$ is abelian.
\end{proof}

\begin{rem}
\end{rem}
Here (following \cite{lyz}) we give another presentation of the 1-cocycle corresponding to a quasi-commutative
structure on a group.
\newline
If $R$ is a Yang-Baxter operator on a group $G$ defining a quasi-commutative structure the corresponding 1-cocycle
$\phi:G\to K$ can be constructed as follows. Write $R(x,y)$ as $(a(x,y),b(x,y))$. As a set $K$ coincides with $G$ while
the multiplication is given by $x*y = xa(x^{-1},y)$. The right $G$-action on $K$ is $x^y = y^{-1}xa(x^{-1},y)$ and the
1-cocycle is the inverse map $\phi(x) = x^{-1}$.
\newline
Indeed, set-theoretically $G^2$ coincides with the product $G\times G$, while the maps $\sigma_0:G^2\to G^1$,
$\partial_0,\partial_1:G^1\to G^2$ have the form: $$\sigma_0(x,y) = xy,\
\partial_0(x) = (x,e),\ \partial_1(x) = (e,x).$$ As the kernel of the multiplication map $K$ can be identified with $G$
via $x\mapsto (x,x^{-1})$. The multiplication $(x,y)*(z,w) = (za(y,z),b(y,z)w)$ on $G\times G$ corresponding to the
Yang-Baxter operator $R$ makes all the above maps group homomorphisms and preserves the image of the anti-diagonal
inclusion $G\to G\times G$ giving the product on $K$: $$(x,x^{-1})*(y,y^{-1}) = (x*y,(x*y)^{-1}).$$ The $G$-action on
$K$ is defined by $$(y,e)^{-1}*(x,x^{-1})*(y,e) = (x^y,(x^y)^{-1}).$$ Finally the 1-cocycle $\phi$ can be defined by
the equation $$(\phi(x),\phi(x)^{-1}) = \partial_0(x)^{-1}\partial_1(x) = (x,e)^{-1}*(e,x) = (x^{-1},x).$$

\begin{rem}
\end{rem}
The full cosimplicial monoid $G^*$ corresponding to a bijective 1-cocycle $\phi:G\to K$ can be described as follows.
The groups $G^n$ can be identified with semi-direct products $G\ltimes K^{\times n-1}$ with respect to the diagonal
action of $G$ on cartesian powers $K^{\times n-1}$. Codegeneration maps are given by
\begin{equation}\label{fcmbc}
\partial_i(g,u_1,...,u_{n-1}) = \left\{ \begin{array}{ll} (g,u_1,...,u_{n-1},e), & i=0 \\
(g,u_1,...,u_{n-i},u_{n-i},...,u_{n-1}), & 0<i<n \\
(g,\phi(g),u_1,...,u_{n-1}), & i=n \end{array}\right.
\end{equation}
while coface maps are given by $$\sigma_j (g,u_1,...,u_{n-1}) = (g,u_1,...,u_{n-j-1},u_{n-j+1},...,u_{n-1}),\quad 0\leq j\leq n-1.$$

\begin{rem}
\end{rem}
Note that a bijective 1-cocycle $\phi:G\to K$ identifies $G$ with $K$ equipped with the product $$u*v =
u^{\phi^{-1}(v)}v,\quad u,v\in K.$$ For an abelian $K$ the (modification of the above) formula $u*v =
u^{\phi^{-1}(v^n)}v$ gives a group structure on $K$ for an arbitrary $n$. Here we identify this structure with the
$n$-th symmetric power of the nearly commutative group $G$.
\newline
Indeed, it follows from the formulas (\ref{fcmbc}) that the maps $\varepsilon_i:G\to G^n,\quad i=1,...,n$,
corresponding to $n$ inclusions of $1$-element set into $n$-element set, have the form $$\begin{array}{lcl}
\varepsilon_0(g) & = & (g,e,...,e), \\ \varepsilon_1(g) & = & (g,\phi(g),e,...,e), \\ & \dots & \\ \varepsilon_n(g) & =
& (g,\phi(g),...,\phi(g)). \end{array}$$ The symmetric group action on $G^n = G\ltimes (K^{\times n-1})$ is uniquely
defined by the conditions $$\varepsilon_i\pi = \varepsilon_{\pi(i)},\quad \forall \pi\in S_n.$$ This, in particular,
allows us to determine the action of the Coxeter generators of $S_n$ on $G^n$: $$\tau_i(g,u_1,...,u_{n-1}) = \left\{
\begin{array}{ll}(g,\phi(g)u_1,...,u_{n-1}), & i=1 \\ (g,u_1,...,u_i,u_iu_{i+1}^{-1}u_{i+2},u_{i+2},...,u_{n-1}), &
0<i<n \\ (g,u_1,...,u_{n-1},u_{n-2}u_{n-1}^{-1}), & i=n \end{array}\right.$$ Thus the subgroup of $S_n$-invariants of
$G^n$ (the $n$-th symmetric power of $G$) is $$S^nG = (G^n)^{S_n} = \{ (\phi^{-1}(u^n),u^{n-1},...,u^2,u),\quad u\in
K\}$$ with the product given by $u*v = u^{\phi^{-1}(v^n)}v$. The 1-cocycle $S^n(\phi)$, corresponding to the nearly
commutative structure on $S^n(G)$, has the form $S^n(\phi)(u) = u^{-1}$.

\newpage
\subsection*{Appendix A. Batanin trees and the free monoidal category generated by a cosimplicial monoid}

We start with a combinatorial description of the category $\overline{\Omega}_2$, which is the full subcategory (of
pruned trees) of the category $\Omega_2$ of trees (of height 2) defined in \cite{ba,bs,jo}. Objects are surjections
(not necessarily order preserving) of finite ordered sets $t:T_2\to T_1$ (``ordered sets of ordered sets"). Morphisms
$(S_1,S_2,s)\to(T_1,T_2,t)$ are pairs of maps $(f_2,f_1)$ forming commutative squares

$$\xymatrix{ S_2 \ar[dd]^s \ar[rr]_{f_2} & & T_2 \ar[dd]^t \\
\\ S_1\ar[rr]^{f_1} & & T_1}$$
where $f_1$ is order-preserving and $f_2$ is order-preserving on fibres (for any $x\in S_1$ the restriction
$f_2:s^{-1}(x)\to t^{-1}(f_1(x))$ is order-preserving). Clearly these pairs are closed under composition. The ordered
sum (union) defines a monoidal structure on the category $\overline{\Omega}_2$ with the unit object given by the
identity map on the empty set $1:\emptyset\to\emptyset$.

There is a full inclusion $S:\Delta\to \overline{\Omega}_2$: $$X\mapsto S(X)=(X\to [1]).$$ Each $S(X)$ is a monoid with
respect to the monoidal structure on $\overline{\Omega}_2$. Indeed the map $$S(X)\otimes S(X) = (X\cup X\to [1]\cup[1]
= [2])\to (X\to [1]) = S(X)$$ given by the pair $(1\cup 1:X\cup X\to X, [2]\to [1])$ is an associative multiplication
with the unit map $(\emptyset\to\emptyset)\to(X\to [1])$. Moreover the morphism $S(X)\to S(Y)$ induced by an
order-preserving map $X\to Y$ is a homomorphism of monoids. Thus the functor $S$ maps into the category
$\Mon(\overline{\Omega}_2)$ of monoids in $\overline{\Omega}_2$ and defines a cosimplicial monoids $S^*$ in
$\overline{\Omega}_2$. We will show that as a monoidal category $\overline{\Omega}_2$ is freely generated by this
cosimplicial monoid.
\begin{prop}
For a monoidal category $\cC$ the evaluation at $S^*$ is an equivalence between the category of monoidal functors
$\Moncat(\overline{\Omega}_2,\cC)$ and the category of cosimplicial monoids in $\cC$.
\end{prop}
\begin{proof}
For a cosimplicial monoid $M^*$ in $\cC$ we will construct a monoidal functor $\overline{\Omega}_2\to\cC$ which sends
$S^*$ in to $M^*$. On objects define it by $([n]\stackrel{t}{\to}[m])\mapsto M^{|t^{-1}(1)|}\otimes...\otimes
M^{|t^{-1}(m)|}$. For a morphism

$$\xymatrix{ [n] \ar[rr]^f \ar[dd]_t & & [l] \ar[dd]^s \\
\\ [m] \ar[rr]^g & & [k] }$$

from $\overline{\Omega}_2$ define the map $$\otimes_{i=1}^m M^{|t^{-1}(i)|} = \otimes_{j=1}^k \otimes_{i\in
g^{-1}(j)}M^{|t^{-1}(i)|} \stackrel{\otimes_{j=1}^k \mu(j)}{\longrightarrow} \otimes_{j=1}^k M^{|s^{-1}(j)|}$$ as the
tensor product of the compositions $$\mu(j):\otimes_{i\in g^{-1}(j)}M^{|t^{-1}(i)|}\stackrel{\otimes_{i\in
g^{-1}(j)}M^{f_i}}{\longrightarrow} \otimes_{i\in g^{-1}(j)}M^{|s^{-1}(j)|} \stackrel{mult}{\longrightarrow}
M^{|s^{-1}(j)|}$$ where $mult$ is an iterated multiplication and $M^{f_i}:M^{|t^{-1}(i)|}\to M^{|s^{-1}(g(i))|}$ is the
map induced by an order preserving restriction $f_i:t^{-1}(i)\to s^{-1}(g(i))$. It is not hard to see that this defines
the desired functor $\overline{\Omega}_2\to\cC$.
\end{proof}

Now we show that covering maps between components of a cosimplicial monoid correspond to morphisms in
$\overline{\Omega}_2$ with bijective second component (morphisms of trees bijective on tips). A covering map is given
by a collection $f_i:[n_i]\to [n]$ of injective order preserving maps such that $im(f_i)\cap im(f_j)=\emptyset$ for
$i\not= j$ and $\cup_i im(f_i) = [n]$. The covering map $$\otimes_{i=1}^m M^{n_i}\stackrel{\otimes_i
M^{f_i}}{\longrightarrow} (M^n)^{\otimes m}\stackrel{mult}{\longrightarrow} M^n$$ corresponds to the composition in
$\overline{\Omega}_2$:

$$\xymatrix{ \cup_{i=1}^m [n_i] \ar[rr]^{\sqcup_i f_i} \ar[dd] & & [m]\times [n] \ar[dd] \ar[rr] & & [n] \ar[dd] \\
\\ [m] \ar[rr] & & [m] \ar[rr] & & [1]}$$
where $[m]\times[n]$ is the product of ordered sets with the lexicographic order. The composition amounts to

\begin{equation}\label{cov}
\xymatrix{ [n] \ar[rr]^f \ar[dd]_t & & [n] \ar[dd] \\
\\ [m] \ar[rr] & & [1]}
\end{equation}
where $t$ is the composition $[n]\simeq \cup_{i=1}^m [n_i] \to [m]$ and $f$ is the composition $$[n]\simeq \cup_{i=1}^m
[n_i] \stackrel{\cup_i f_i}{\longrightarrow} [n]$$ which is bijective by the covering condition.

Conversely any morphism in $\overline{\Omega}_2$

\begin{equation}\label{bot}
\xymatrix{ [n] \ar[rr]^f \ar[dd]_t & & [n] \ar[dd]^s \\
\\ [m] \ar[rr] & & [k]}
\end{equation}
with bijective $f$ is a tensor product of morphisms of the form (\ref{cov}). First, $g$ is surjective since $g\circ t =
s\circ f$ is. Second, $[n]\stackrel{f}{\longrightarrow} [m]$ is the tensor product $\otimes_{i=1}^k
([|(gt)^{-1}(i)|]\to [|g^{-1}(i)|])$. Finally the morphism (\ref{bot}) is the tensor product of
$(f_i,1):([|(gt)^{-1}(i)|]\to [|g^{-1}(i)|])\to ([|(gt)^{-1}(i)|]\to [1])$ where $f_i = f|_{(gt)^{-1}(i)}$ are the
restrictions of $f$.

Using the language of trees we can reformulate the main theorem \ref{main} as follows. Denote by ${\cal Q}$ the
collection of morphisms in $\overline{\Omega}_2$ with bijective second component (bijective on tips). The fact that the
only covering condition needs to be checked on level 2 has the following interpretation in terms of maps bijective on
tips.
\begin{lem}
The collection ${\cal Q}$ is monoidally generated by (every element is a composition of tensor products of) the
identity morphisms and the morphisms

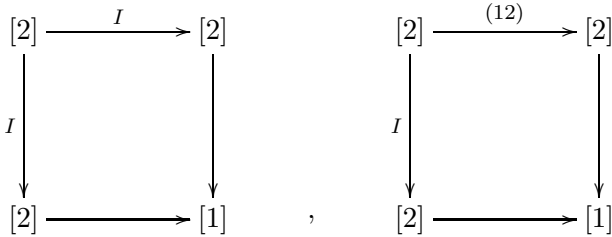
\xymatrix{ [2] \ar[rr]^I \ar[dd]_I & & [2] \ar[dd] &&  [2] \ar[rr]^{(12)} \ar[dd]_I & & [2] \ar[dd]\\
\\ [2] \ar[rr] & & [1] &,& [2] \ar[rr] & & [1]}
\end{lem}
\begin{proof}
This follows directly from the fact that bijective maps of the monoidal category $\bS$ of finite sets are monoidally
generated by identity morphisms and the transposition $(12):[2]\to[2]$.
\end{proof}

\begin{theo}
The category of fractions $\overline{\Omega}_2[{\cal Q}^{-1}]$ is  the free braided monoidal category generated by a commutative monoid.
\end{theo}
\begin{proof}

The proof is absolutely similar to the proof of theorem \ref{main} and consists of establishing a structure of a
quasi-commutative monoid on $S([1])$.  We will use graphical presentation of trees, which assigns to a surjection
$t:T_2\to T_1$ a planar tree of height 2 with one root node (at the height 0), with height 1 nodes labeled by $T_1$ and
top (height 2) nodes labelled by $T_2$. All height 1 nodes are connected with the root, a height 2 node is connected
with a height 1 node if the map $t$ maps the first into the second (see \cite{ba,bs} for details). In particular, the
trees corresponding to $$S([1]), S([1])\otimes S([1]), S([2]), S([1])\otimes S([1])\otimes S([1]), S([2])\otimes
S([1]), S([1])\otimes S([2])$$ have the following form: $$\TreeOne\quad,\quad \TreeOneOne\quad,\quad
\TreeTwo\quad,\quad \TreeOneOneOne\quad,\quad \TreeTwoOne\quad,\quad  \TreeOneTwo .$$ Inverting morphisms bijective on
tips (height 2 nodes) allows to define an automorphism $R$: 
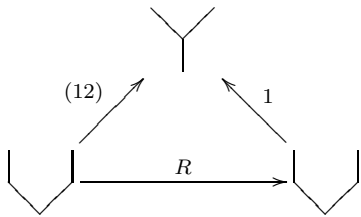
$$\xymatrix{ & \TreeTwo & \\ \TreeOneOne \ar[rr]^R
\ar[ur]^{(12)} && \TreeOneOne \ar[ul]_1 }$$ Here we indicate only the top components of the morphisms (height one
components are determined uniquiely), so $1$ is the identity on tips and $(12)$ is the transposition of tips. As in the
proof of the theorem \ref{main} we use commutative diagrams to prove the defining relations of the quasi-commutative
structure. This time the language of trees allows us to make them more compact.

Condition (\ref{com}) follows from the diagram: $$\xygraph{ !{0;/r4.0pc/:;/u4.0pc/::}[]*+{\TreeOneOne}(
  ( :[drr]*++{\TreeOne}="r" ^\mu
  , :[rd]*+{\TreeTwo}="m" _{(12)}
    :"r"
  , :[dd]*+{\TreeOneOne} _R
    ( :"m" ^1
    , :"r" _\mu
    )
  )
) }$$ Here $\mu$ is the uniquely defined morphism, which is the multiplication map of $S([1])$. Conditions (\ref{ide})
are guaranteed by the diagrams:

$$
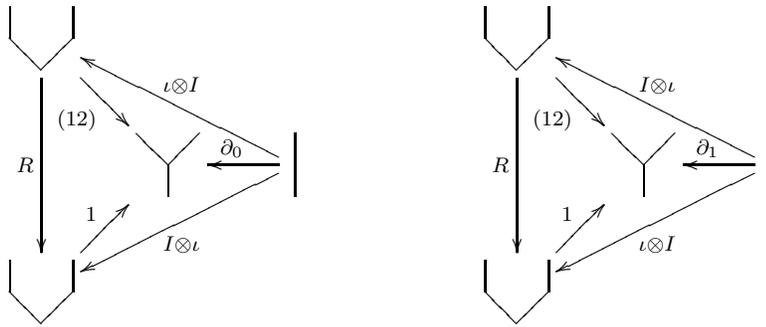
\begin{xy} (-10,0)*+!R{\xybox{\xygraph{ !{0;/r4.0pc/:;/u4.0pc/::}[]*++{\TreeOne}(
  ( :[ull]*+{\TreeOneOne} _{\iota\otimes I}
    ( :[rd]*+{\TreeTwo}="m" _{(12)}
    ,
      :[dd]*+{\TreeOneOne}="d" _R
      :"m" ^1
    )
  , :"m" _{\partial_0}
  , :"d" ^{I\otimes\iota}
  )
) }}}
\POS(10,0)*+!L{\xybox{\xygraph{ !{0;/r4.0pc/:;/u4.0pc/::}[]*++{\TreeOne}(
  ( :[ull]*+{\TreeOneOne} _{I\otimes\iota}
    ( :[rd]*+{\TreeTwo}="m" _{(12)}
    ,
      :[dd]*+{\TreeOneOne}="d" _R
      :"m" ^1
    )
  , :"m" _{\partial_1}
  , :"d" ^{\iota\otimes I}
  )
) }}}\end{xy}$$ Here $\iota$ is the unit of $S([1])$ and $\partial_0,\partial_1$ are the morphisms mapping the tip of
$S([1])$ into the left (right) tip of $S([2])$ respectively.

For condition (\ref{mur}) we have the following commutative diagram: $$\xygraph{
!{0;/r6.0pc/:;/u4.0pc/::}[]*+{\TreeOneOneOne}(
  ( :[dr]*+{\TreeTwoOne}="mu" _{(12)}
    :[dr]*+{\TreeTwo}="r" _{ \left(\begin{smallmatrix} 1 & 2 & 3 \\ 1 & 2 & 1 \end{smallmatrix}\right)}
  , :[ul]*+{\TreeOneOne} ^{1\otimes\mu}
    ( :@/^6ex/ "r" ^{(12)}
    , :[d(6)]*+{\TreeOneOne}="d" _R
      :@/_6ex/ "r" _1
    )
  , :[dd]*+{\TreeOneOneOne} _{R\otimes 1}
    ( :"mu" ^1
    , :[dr]*+{\TreeOneTwo}="md" _{(23)}
      :"r" ^{\left(\begin{smallmatrix} 1 & 2 & 3 \\ 1 & 1 & 2 \end{smallmatrix}\right) }
    , :[dd]*+{\TreeOneOneOne} _{1\otimes R}
      ( :"md" ^1
      , :"d" _{\mu\otimes 1}
      )
    )
  )
) }$$ Analogously condition (\ref{mur}) follows from: $$\xygraph{ !{0;/r6.0pc/:;/u4.0pc/::}[]*+{\TreeOneOneOne}(
  ( :[dr]*+{\TreeOneTwo}="mu" _{(23)}
    :[dr]*+{\TreeTwo}="r" _{\left(\begin{smallmatrix} 1 & 2 & 3 \\ 2 & 1 & 2 \end{smallmatrix}\right) }
  , :[ul]*+{\TreeOneOne} ^{\mu\otimes 1}
    ( :@/^6ex/ "r" ^{(12)}
    , :[d(6)]*+{\TreeOneOne}="d" _R
      :@/_6ex/ "r" _1
    )
  , :[dd]*+{\TreeOneOneOne} _{1\otimes R}
    ( :"mu" ^1
    , :[dr]*+{\TreeTwoOne}="md" _{(12)}
      :"r" ^{\left(\begin{smallmatrix} 1 & 2 & 3 \\ 1 & 2 & 2 \end{smallmatrix}\right) }
    , :[dd]*+{\TreeOneOneOne} _{R\otimes 1}
      ( :"md" ^1
      , :"d" _{1\otimes\mu}
      )
    )
  )
) }$$
As before we only write the effect of morphisms on tips.

Finally for the Yang-Baxter equation we have the following commutative diagram:
$$\xygraph{ !{0;/r4.7pc/:;/u4.0pc/::}[]*+{\TreeOneOneOne}(
  ( :[ddl]*+{\TreeOneTwo}="(-1,2)" ^{(23)}
    :[ddr]*{\TreeThree}="(0,0)" ^{(132)}
  , :[ddr]*+{\TreeTwoOne}="(1,2)" _{(12)}
    :"(0,0)" _{(123)}
  , :[ddrrr]*+{\TreeOneOneOne} ^{R\otimes 1}
    ( :"(1,2)" _1
    , :[ddl]*+{\TreeOneTwo}="(2,0)" _{(23)}
      :"(0,0)" _{(12)}
    , :[d(4)]*+{\TreeOneOneOne} ^{1\otimes R}
      ( :"(2,0)" ^1
      , :[ll]*+{\TreeTwoOne}="(1,-2)" _{(12)}
        :"(0,0)" ^1
      , :[ddlll]*+{\TreeOneOneOne}="(0,-4)" ^{R\otimes 1}
        ( :"(1,-2)" ^1
        , :[uul]*+{\TreeOneTwo}="(-1,-2)" _1
          :"(0,0)" _1
        )
      )
    )
  , :[ddlll]*+{\TreeOneOneOne} _{1\otimes R}
    ( :"(-1,2)" ^1
    , :[ddr]*+{\TreeTwoOne}="(-2,0)" ^{(12)}
      :"(0,0)" ^{(23)}
    , :[d(4)]*+{\TreeOneOneOne} _{R\otimes 1}
      ( :"(-2,0)" _1
      , :"(-1,-2)" ^{(23)}
      , :"(0,-4)" _{1\otimes R}
      )
    )
  )
) }$$
Again the labels for morphisms indicate the effect on tips.
\end{proof}


\begin{thebibliography}{11}
\bibitem{bz}
J. Baez,
Hochschild homology in a braided tensor category.
Trans. Amer. Math. Soc. 344 (1994), no. 2, 885--906.
%
\bibitem{ba}
M. Batanin, Monoidal globular categories as a natural environment for the theory of weak $n$-categories. Adv.
Math. 136 (1998), no. 1, 39--103.
%
\bibitem{bs}
M. Batanin, R. Street, The universal property of the multitude of trees.  Category theory and its applications
(Montreal, QC, 1997). J. Pure Appl. Algebra 154 (2000), no. 1-3, 3--13.
%
\bibitem{ds}
B. Day, R. Street,
Abstract substitution in enriched categories.
J. Pure Appl. Algebra 179 (2003), no. 1-2, 49--63.
%
\bibitem{dr}
V. G. Drinfel'd, On some unsolved problems in quantum group theory. Quantum groups (Leningrad, 1990), 1-8, Lecture
Notes in Math., 1510, Springer, Berlin, 1992.
%
\bibitem{ess}
P. Etingof, T. Schedler, A. Soloviev, Set-theoretical solutions to the quantum Yang-Baxter equation. Duke Math. J.
100 (1999), no. 2, 169--209.
%
\bibitem{jo}
A. Joyal, Disks, duality and $\Theta$-categories, preprint and talk at the AMS Meeting in Motreal, September 1997.
%
\bibitem{js}
A. Joyal, R. Street, The geometry of tensor calculus. I. Adv. Math. 88 (1991), no. 1, 55--112.
%
\bibitem{la}
T. G. Lavers, The theory of vines. Comm. Algebra 25 (1997), no. 4, 1257--1284.
%
\bibitem{ma}
Yu. Manin,
Quantum groups and non-commutative differential geometry. Mathematical physics, X (Leipzig, 1991), 113--122, Springer, Berlin, 1992.
%
\bibitem{lw}
W. Lawvere,  Functorial semantics of algebraic theories. Proc. Nat. Acad. Sci. U.S.A. 50 1963 869--872.
%
\bibitem{lo}
J.-L. Loday, Cyclic homology. Grundlehren der Mathematischen Wissenschaften, 301. Springer-Verlag, Berlin, 1992.
454 pp.
%
\bibitem{lyz}
J.-H. Lu, M. Yan, Y.-C. Zhu, On the set-theoretical Yang-Baxter equation. Duke Math. J. 104 (2000), no. 1, 1--18.
%
\bibitem{mc0}
S. MacLane,
Categorical algebra. Bull. Amer. Math. Soc. 71 1965 40--106.
%
\bibitem{mc}
S. MacLane, Categories for the working mathematician. Graduate Texts in Mathematics, Vol. 5. Springer-Verlag, New
York-Berlin, 1971. 262 pp.
%
\bibitem{mon}
S. Montgomery,
Hopf algebras and their actions on rings.  American Mathematical Society, Providence, RI, 1993. 238 pp.
%
\bibitem{st}
R. Street, Higher categories, strings, cubes and simplex equations. Appl. Categ. Structures 3 (1995), no. 1, 29--77.
%
\bibitem{sw}
M. Sweedler,
Hopf algebras. New York 1969 336 pp.
%
\end{thebibliography}
\end{document}